\documentclass[12pt]{article}
\usepackage[top=1in,bottom=1in,left=1in,right=1in]{geometry}
\usepackage[T1]{fontenc}
\usepackage{lmodern}
\usepackage{microtype}
\usepackage{indentfirst}
\usepackage{amsfonts,amsmath,amsthm,amssymb,mathtools,bm}
\usepackage{enumitem}
\usepackage{needspace}
\usepackage{longtable}
\usepackage[caption=false]{subfig}
\usepackage{leftidx}
\usepackage{lineno}
\usepackage{color,xcolor}
\usepackage{array}
\usepackage{latexsym}
\usepackage{float}
\usepackage{supertabular,booktabs}
\usepackage{rotating}
\usepackage{cases}
\usepackage{multicol,multirow}
\usepackage{epsfig}
\usepackage{cite}
\usepackage{fancyhdr}
\usepackage{dsfont}
\usepackage{enumitem}
\usepackage{hyperref}

\hypersetup{
	colorlinks=true,
	linkcolor=red,
	filecolor=blue,
	urlcolor=red,
	citecolor=blue,
	pdftitle={The Matrix Pythagorean Equation over GL2(Z)},
	pdfauthor={Hongjian Li},
	pdfsubject={The matrix Pythagorean equation over GL2(Z)},
	pdfkeywords={matrix Pythagorean equation, GL2Z, SL2Z, orbit decomposition, integral conjugacy, canonical representatives}
}

\allowdisplaybreaks
\numberwithin{equation}{section}

\newtheorem{theorem}{Theorem}[section]

\newtheorem{lemma}{Lemma}[section]

\newtheorem{corollary}{Corollary}[section]
\newtheorem{remark}{Remark}[section]

\def\G1{G^\mathcal{C}}

\newcommand{\SL}{{\rm SL}}
\newcommand{\GL}{{\rm GL}}

\newcommand{\tr}{\operatorname{tr}}

\begin{document}
\title{The Matrix Pythagorean Equation over $\GL_2(\mathbb Z)$}
\author{
Hongjian Li$^{1,}$\footnote{E-mail\,$:$ lhj@gdufs.edu.cn. Supported by the Project of Guangdong University of Foreign Studies (Grant No. 2024RC063).}\quad
Weilin Zhang$^{2,}$\footnote{Corresponding author. E-mail\,$:$ weilin@m.scnu.edu.cn.}\\
{\small\it $^{1}$School of Mathematics and Statistics, Guangdong University of Foreign Studies,}\\
{\small\it Guangzhou 510006, Guangdong, P. R. China}\\
{\small\it $^{2}$School of Mathematics and Information Science, Guangzhou University,}\\
{\small\it Guangzhou 510006, Guangdong, P. R. China}
}
\date{}
\maketitle
\date{}

\noindent{\bf Abstract}\quad
In this paper, we study the ordered solutions of the matrix Pythagorean equation $X^2+Y^2=Z^2$ over $\GL_2(\mathbb Z)$. Exploiting the independent sign symmetry of the equation, we first reduce the full solution set to the trace-nonnegative subset $\mathcal M$, consisting of those solutions for which all three traces are nonnegative. We then determine a canonical decomposition of $\mathcal M$ into orbits under simultaneous integral conjugacy.


\medskip
\noindent{\bf MR(2020) Subject Classification} 15A24; 15B36; 11D09; 20H05

\section{Introduction}

The classical Fermat equation is
\[
x^n+y^n=z^n,\qquad n\geq3.
\]
Replacing the integer variables by integral matrices leads naturally to the matrix analogue
\[
X^n+Y^n=Z^n.
\]
Early work on matrix Fermat equations includes that of Domiaty \cite{Domiaty1966}, Bolker \cite{Bolker1968}, Gibson \cite{Gibson1970}, and Brenner and de Pillis \cite{BrennerDePillis1972}. Vaserstein subsequently placed related additive and power equations in a broader noncommutative number-theoretic framework \cite{Vaserstein1989}, while Khazanov investigated Fermat equations over $\SL_2(\mathbb Z)$ and related matrix groups \cite{Khazanov1995}. Further results concerning integral $2\times2$ matrices, matrix rings, and matrix groups were obtained by Le and Li \cite{LeLi1995}, Grytczuk \cite{Grytczuk1995}, Qin \cite{Qin1996}, and Patay and Szak\'acs \cite{PataySzakacs2002}. More recently, Chien and Meng studied several structured classes of $2\times2$ matrices \cite{ChienMeng2021}, Li and Yuan established solvability criteria for Fermat and Catalan equations in centralizers of integral $2\times2$ matrices \cite{LiYuan2023}, and Sarma gave a general construction over $M_k(\mathbb Z)$ \cite{Sarma2025}. For background on matrix polynomials and integral matrices, see \cite{GohbergLancasterRodman1982,Newman1972}.

The case of exponent $2$ is particularly natural. The classical Pythagorean equation $x^2+y^2=z^2$ is one of the fundamental Diophantine equations, and its matrix analogue
\[
X^2+Y^2=Z^2
\]
combines arithmetic constraints with genuinely noncommutative phenomena. Arnold and Eydelzon \cite{ArnoldEydelzon2019} constructed a parametric family of integral matrix Pythagorean triples, while explicitly noting that their parametrization does not cover all possible solutions. Thus, the problem of obtaining an exhaustive description is distinct from that of constructing parametric families. In contrast with the scalar setting, classical factorization and parametrization arguments do not carry over directly, since matrix multiplication is generally noncommutative. In the present paper, we restrict attention to $X,Y,Z\in\GL_2(\mathbb Z)$. Within this unimodular $2\times2$ setting, the possible squares of the matrices are strongly constrained by trace and determinant identities, congruence conditions, and integral conjugacy. These additional restrictions provide sufficient rigidity to obtain a complete classification of the ordered solutions over $\GL_2(\mathbb Z)$. For convenience, we let $G=\GL_2(\mathbb Z)$ throughout this paper. Since the equation involves only the squares of $X$, $Y$, and $Z$, it is invariant under independent sign changes:
\[
(X,Y,Z)\longmapsto
(\varepsilon_1X,\varepsilon_2Y,\varepsilon_3Z),
\qquad
\varepsilon_1,\varepsilon_2,\varepsilon_3\in\{\pm1\}.
\]
Moreover, $\tr(-X)=-\tr(X)$, and similarly for $Y$ and $Z$. Hence, for every ordered solution $(X,Y,Z)\in G^3$, one can choose $\varepsilon_1,\varepsilon_2,\varepsilon_3\in\{\pm1\}$ so that $\tr(\varepsilon_1X),\,\tr(\varepsilon_2Y),\,\tr(\varepsilon_3Z)\ge0$. Accordingly, we consider the trace-nonnegative solution set
\[
\mathcal M:=
\{(X,Y,Z)\in G^3:
X^2+Y^2=Z^2,\
\tr(X),\tr(Y),\tr(Z)\ge0\}.
\]
Conversely, every ordered solution in $G^3$ can be recovered from a solution in $\mathcal M$ by independent sign changes. Thus, it suffices to classify $\mathcal M$.

We next consider the natural action of $G$ on $\mathcal M$ by simultaneous integral conjugation. For $P\in G$ and $(X,Y,Z)\in\mathcal M$, define
\[
P\cdot(X,Y,Z)
:=
(PXP^{-1},\,PYP^{-1},\,PZP^{-1}).
\]
This map is well defined. Indeed, since $G$ is closed under conjugation, we have
\[
PXP^{-1},\,PYP^{-1},\,PZP^{-1}\in G.
\]
Moreover,
\[
(PXP^{-1})^2+(PYP^{-1})^2
=
P(X^2+Y^2)P^{-1}
=
PZ^2P^{-1}
=
(PZP^{-1})^2,
\]
and trace is invariant under conjugation. Hence, $P\cdot(X,Y,Z)\in\mathcal M$ for all $P\in G$ and $(X,Y,Z)\in\mathcal M$. Finally, the identity element acts trivially and
\[
(PQ)\cdot(X,Y,Z)
=
P\cdot\bigl(Q\cdot(X,Y,Z)\bigr),
\qquad P,Q\in G,
\]
so this indeed defines a group action of $G$ on $\mathcal M$. For $(X,Y,Z)\in\mathcal M$, we denote its orbit under this action by
\[
\mathcal O_G(X,Y,Z)
:=
\left\{
(PXP^{-1},\,PYP^{-1},\,PZP^{-1})
:\ P\in G
\right\}.
\]
The action of $G$ partitions $\mathcal M$ into pairwise disjoint orbits:
\[
\mathcal M
=
\mathcal O_G(X_1,Y_1,Z_1)
\sqcup
\mathcal O_G(X_2,Y_2,Z_2)
\sqcup\cdots,
\]
where $\mathcal O_G(X_1,Y_1,Z_1),\,\mathcal O_G(X_2,Y_2,Z_2),\,\ldots$ are all the distinct $G$-orbits in $\mathcal M$. Here and throughout this paper, the symbol $\sqcup$ denotes the disjoint union. The main result of this paper gives an explicit decomposition of $\mathcal M$ into $G$-orbits. More precisely, we prove the following theorem.

\begin{theorem}\label{thm:orbit-main}
Let $R=\begin{pmatrix}0&-1\\1&-1\end{pmatrix},\,
U=\begin{pmatrix}-1&1\\-1&2\end{pmatrix},\,
V=\begin{pmatrix}2&-1\\3&-1\end{pmatrix},\,
W=\begin{pmatrix}1&0\\1&1\end{pmatrix}$, and let
\[
\mathcal T
=\left\{\begin{pmatrix}
 a&-a-c\\
 a+b&-a
\end{pmatrix}:(a,b,c)\in\mathbb Z^3,ab+bc+ca=1,a\le b,a<c\right\}.
\]
Then
\begin{equation}\label{eq:main-orbit-decomposition}
\mathcal M
=
\mathop{\sqcup}\limits_{T\in\mathcal T}
\mathcal O_G(-R^2,-R,T)
\sqcup
\mathcal O_G(U,V,W)
\sqcup
\mathcal O_G(V,U,W).
\end{equation}
\end{theorem}

The paper is organized as follows. Section~\ref{sec:structural} establishes the auxiliary results needed for the proof. Section~\ref{sec:main-proof} proves Theorem~\ref{thm:orbit-main}.

\section{Preliminaries}\label{sec:structural}

Throughout the remainder of the paper, we retain the notation $G,\mathcal M,R,U,V,W$, and $\mathcal T$ from the Introduction. We also set
\[
J:=\begin{pmatrix}0&-1\\1&0\end{pmatrix}.
\]
Then
\[
R^2=\begin{pmatrix}-1&1\\-1&0\end{pmatrix},\qquad
R^2+R+I=0,\qquad R^3=I,\qquad J^2=-I.
\]

We first recall the integral-conjugacy facts needed below.

\begin{lemma}[Latimer--MacDuffee~\cite{Newman1972}, Theorem~III.13]\label{lem:latimer-macduffee}
Let $f(x)\in\mathbb Z[x]$ be a monic polynomial of degree $n$ that is irreducible over $\mathbb Q$, and let $\theta$ be a root of $f(x)$. Then there is a one-to-one correspondence between the $\GL_n(\mathbb Z)$-similarity classes of matrices $A\in M_n(\mathbb Z)$ satisfying $f(A)=0$ and the ideal classes of the ring $\mathbb Z[\theta]$.
\end{lemma}

\begin{lemma}\label{lem:order3-conj-elementary}
Let $M\in M_2(\mathbb Z)$ satisfy $M^2+M+I=0$. Then there exists $P\in \GL_2(\mathbb Z)$ such that $P^{-1}MP=R$.
\end{lemma}
\begin{proof}
Let $f(x)=x^2+x+1$ and $\omega=\frac{-1+\sqrt{-3}}2$. Then $f(x)$ is irreducible over $\mathbb Q$, and $\mathbb Z[\omega]$ is the ring of Eisenstein integers. This ring is Euclidean with respect to the norm $N(a+b\omega)=a^2-ab+b^2$, and hence is a principal ideal domain \cite{Cohen1993}. Therefore, $\mathbb Z[\omega]$ has only one ideal class. By assumption, $f(M)=M^2+M+I=0$. Moreover, since $R^2+R+I=0$, we have $f(R)=0$. It follows from Lemma~\ref{lem:latimer-macduffee} that $M$ and $R$ belong to the same $\GL_2(\mathbb Z)$-similarity class. Hence there exists $P\in\GL_2(\mathbb Z)$ such that $P^{-1}MP=R$.
\end{proof}

\begin{lemma}\label{lem:minusI-conjugacy}
Let $W\in M_2(\mathbb Z)$ satisfy $W^2=-I$. Then there exists $Q\in \GL_2(\mathbb Z)$ such that $Q^{-1}WQ=J$.
\end{lemma}

\begin{proof}
Let $g(x)=x^2+1$. Then $g(x)$ is irreducible over $\mathbb Q$, and $\mathbb Z[i]$ is the ring of Gaussian integers. This ring is Euclidean with respect to the norm $N(a+bi)=a^2+b^2$, and hence is a principal ideal domain \cite{Cohen1993}. Therefore, $\mathbb Z[i]$ has only one ideal class. By assumption, $g(W)=W^2+I=0$. Moreover, since $J^2=-I$, we have $g(J)=0$. It follows from Lemma~\ref{lem:latimer-macduffee} that $W$ and $J$ belong to the same $\GL_2(\mathbb Z)$-similarity class. Hence there exists $Q\in\GL_2(\mathbb Z)$ such that $Q^{-1}WQ=J$.
\end{proof}

The following lemma gives a useful relation between a square root in $G$ and its trace and determinant.

\begin{lemma}\label{lem:square-root}
Let $M\in\SL_2(\mathbb Z)$ and let $W\in G$ satisfy $W^2=M$. Then
\[
M+(\det W)I=\tr(W)W,
\qquad
\tr(M)=\tr(W)^2-2\det W.
\]
\end{lemma}

\begin{proof}
By the Cayley--Hamilton identity,
\[
W^2-\tr(W)W+\det(W)I=0.
\]
Since $W^2=M$, we obtain
\[
M+(\det W)I=\tr(W)W.
\]
Taking traces gives
\[
\tr(M)=\tr(W)^2-2\det W.
\]
\end{proof}

We define
\[
\mathcal S:=\left\{(X,Y,Z)\in G^3:\begin{array}{l}
X^2=A,\,Y^2=AR,\,Z^2=-AR^2\,\text{ for some }A\in\SL_2(\mathbb Z),\\
\tr(X),\tr(Y),\tr(Z)\ge0
\end{array}\right\}.
\]
The following lemma describes the relation between $\mathcal S$ and $\mathcal M$.

\begin{lemma}\label{lem:orbit-reduction}
We have
\[
\mathcal M=\bigcup_{(X,Y,Z)\in\mathcal S}\mathcal O_G(X,Y,Z).
\]
\end{lemma}

\begin{remark}
The union in Lemma~\ref{lem:orbit-reduction} is an ordinary union and need not be disjoint, since distinct elements of $\mathcal S$ may belong to the same $G$-orbit. Here and throughout the remainder of the paper, the symbol $\cup$ denotes an ordinary union, whereas $\sqcup$ denotes a disjoint union.
\end{remark}

\begin{proof}
Let $(X,Y,Z)\in\mathcal S$. Then for some $A\in\SL_2(\mathbb Z)$, we have
\[
X^2=A,\qquad
Y^2=AR,\qquad
Z^2=-AR^2.
\]
Since $R^2+R+I=0$, we have
\[
X^2+Y^2
=
A+AR
=
A(I+R)
=
-AR^2
=
Z^2.
\]
Since $\tr(X),\tr(Y),\tr(Z)\ge0$ by the definition of $\mathcal S$, it follows that $(X,Y,Z)\in\mathcal M$. Since simultaneous integral conjugation preserves the equation and the traces, it follows that
\[
\mathcal O_G(X,Y,Z)\subseteq\mathcal M.
\]
Hence
\[
\bigcup_{(X,Y,Z)\in\mathcal S}
\mathcal O_G(X,Y,Z)
\subseteq\mathcal M.
\]

Conversely, let $(X,Y,Z)\in\mathcal M$, and set
\[
A_0=X^2,\qquad B_0=Y^2,\qquad C_0=Z^2.
\]
Then $A_0,B_0,C_0\in\SL_2(\mathbb Z)$ and $A_0+B_0=C_0$. Put $T=A_0^{-1}B_0\in\SL_2(\mathbb Z)$. Since $C_0=A_0(I+T)$ and $\det A_0=\det C_0=1$, we obtain $\det(I+T)=1$. For a $2\times2$ matrix,
\[
\det(I+T)=1+\tr(T)+\det(T).
\]
Since $\det(T)=1$, it follows that $\tr(T)=-1$. By the Cayley--Hamilton theorem,
\[
T^2+T+I=0.
\]
Lemma~\ref{lem:order3-conj-elementary} therefore yields a matrix $P\in G$ such that $P^{-1}TP=R$. Set
\begin{equation}\label{eq1}
X_1=P^{-1}XP,\qquad
Y_1=P^{-1}YP,\qquad
Z_1=P^{-1}ZP,
\qquad
A=P^{-1}A_0P.
\end{equation}
Then $A\in\SL_2(\mathbb Z)$ and
\[
X_1^2
=
(P^{-1}XP)^2
=
P^{-1}X^2P
=
P^{-1}A_0P
=
A,
\]
while
\[
Y_1^2
=
(P^{-1}YP)^2\\
=
P^{-1}Y^2P\\
=
P^{-1}B_0P\\
=
P^{-1}A_0TP\\
=
(P^{-1}A_0P)(P^{-1}TP)\\
=
AR.
\]
Moreover,
\[
\begin{aligned}
Z_1^2&=(P^{-1}ZP)^2=P^{-1}Z^2P=P^{-1}C_0P=P^{-1}(A_0+B_0)P\\
&=P^{-1}A_0P+P^{-1}B_0P=A+AR=A(I+R)=-AR^2.
\end{aligned}
\]
Since trace is invariant under conjugation, $\tr(X_1),\tr(Y_1),\tr(Z_1)\ge0$. Thus $(X_1,Y_1,Z_1)\in\mathcal S$. By \eqref{eq1}, we have $(X,Y,Z)=P\cdot(X_1,Y_1,Z_1)$. So $(X,Y,Z)\in\mathcal O_G(X_1,Y_1,Z_1)$. Therefore
\[
\mathcal M
\subseteq
\bigcup_{(X,Y,Z)\in\mathcal S}
\mathcal O_G(X,Y,Z),
\]
which proves the result.
\end{proof}

We next establish some basic properties of $\mathcal S$.

\begin{lemma}\label{lem:trace-divisibility}
Let $(X,Y,Z)\in\mathcal S$. Then the following statements hold:
\begin{enumerate}
\item[\textup{(i)}] $\tr(X)$ and $\tr(Y)$ are odd, whereas $\tr(Z)$ is even.

\item[\textup{(ii)}] If $\tr(Z)=0$, then $\det X=\det Y=\det Z=1$ and $X^2=R$.

\item[\textup{(iii)}] If $\tr(Z)\ne0$, set $t=\tr(X)^2-2\det X$ and $u=\tr(Y)^2-2\det Y$. Then there exist $m,n\in\mathbb Z$ such that
\[
\bigl(\tr(X)\tr(Y)\tr(Z)\bigr)^2(3m^2+n^2)
=
4\bigl(t^2+tu+u^2-3\bigr).
\]
\end{enumerate}
\end{lemma}

\begin{proof}
Choose $A\in\SL_2(\mathbb Z)$ such that
\[
X^2=A,\qquad Y^2=AR,\qquad Z^2=-AR^2,
\]
and put $p=\tr(X)$, $q=\tr(Y)$, and $r=\tr(Z)$. Since $X,Y,Z\in G$,
we have $\det X,\det Y,\det Z\in\{\pm1\}$. Moreover,
$A,AR,-AR^2\in\SL_2(\mathbb Z)$. By Lemma~\ref{lem:square-root},
\[
A+(\det X)I=pX,\qquad
AR+(\det Y)I=qY,\qquad
-AR^2+(\det Z)I=rZ.
\]
Taking traces gives
$\tr(A)=p^2-2\det X$, $\tr(AR)=q^2-2\det Y$, and
$\tr(-AR^2)=r^2-2\det Z$.

Since $R^2+R+I=0$, we have $-AR^2=A+AR$. Hence
\begin{equation}\label{eq:trace-key-relation}
p^2+q^2-r^2
=
2(\det X+\det Y-\det Z).
\end{equation}
Since $\det X,\det Y,\det Z\in\{\pm1\}$, the integer $\det X+\det Y-\det Z$ is odd. Hence the right-hand side of \eqref{eq:trace-key-relation} is congruent to $2$ modulo $4$. Since every square is congruent to $0$ or $1$ modulo $4$, it follows that
\[
p^2\equiv q^2\equiv1\pmod4,
\qquad
r^2\equiv0\pmod4.
\]
Thus $p$ and $q$ are odd, whereas $r$ is even, proving
\textup{(i)}. Since $p,q\ge0$, we also have $p,q>0$.

For \textup{(ii)}, suppose that $r=0$. From
$-AR^2+(\det Z)I=rZ$ we obtain $AR^2=(\det Z)I$, and hence, using
$R^3=I$,
\[
A=(\det Z)R.
\]
If $\det Z=-1$, then $A=-R$ and $\tr(A)=1$. Thus
$1=p^2-2\det X$, which gives either $p^2=3$ or $p^2=-1$ according as
$\det X=1$ or $-1$, a contradiction. Therefore $\det Z=1$ and $A=R$.

Since $\tr(R)=-1$, we have $-1=p^2-2\det X$, whence
$\det X=1$ and $p^2=1$. Also $AR=R^2$ and $\tr(R^2)=-1$, so
$-1=q^2-2\det Y$, giving $\det Y=1$ and $q^2=1$. Therefore
\[
\det X=\det Y=\det Z=1,
\qquad
X^2=A=R.
\]
This proves \textup{(ii)}.

Finally, suppose that $r\ne0$. By \textup{(i)} and the definition of
$\mathcal S$, we have $p,q,r>0$. The three square-root identities imply
\[
A\equiv-(\det X)I\pmod p,\quad
AR\equiv-(\det Y)I\pmod q,\quad
-AR^2\equiv-(\det Z)I\pmod r.
\]
Write $A=\begin{pmatrix}a&b\\c&d\end{pmatrix}$. Since
\[
AR=\begin{pmatrix}b&-a-b\\ d&-c-d\end{pmatrix},
\qquad
-AR^2=\begin{pmatrix}a+b&-a\\ c+d&-c\end{pmatrix},
\]
the preceding congruences yield
\[
\begin{array}{rcl@{\qquad}rcl@{\qquad}rcl@{\qquad}rcl@{\qquad}l}
a&\equiv&-\det X, &
b&\equiv&0, &
c&\equiv&0, &
d&\equiv&-\det X
&(\bmod\,p),\\
a&\equiv&\det Y, &
b&\equiv&-\det Y, &
c&\equiv&\det Y, &
d&\equiv&0
&(\bmod\,q),\\
a&\equiv&0, &
b&\equiv&-\det Z, &
c&\equiv&\det Z, &
d&\equiv&-\det Z
&(\bmod\,r).
\end{array}
\]

By Lemma~\ref{lem:square-root},
$t=\tr(A)=a+d$ and $u=\tr(AR)=b-c-d$. Hence
$d=t-a$ and $c=a+b-t-u$. Substituting these into
$\det A=ad-bc=1$ gives
\[
a^2+ab+b^2-ta-(t+u)b+1=0.
\]
Set $\lambda=2a+b-t$ and $\mu=3b-t-2u$. A direct calculation gives
\[
3\lambda^2+\mu^2
-4(t^2+tu+u^2-3)
=
12\bigl(a^2+ab+b^2-ta-(t+u)b+1\bigr),
\]
and therefore
\[
3\lambda^2+\mu^2
=
4(t^2+tu+u^2-3).
\]

We first show that each of $p,q,r$ divides both $\lambda$ and $\mu$. Indeed, the above congruences give
\[
\begin{aligned}
t&\equiv -2\det X, &\qquad u&\equiv \det X      &&(\bmod\,p),\\
t&\equiv \det Y,   &\qquad u&\equiv -2\det Y   &&(\bmod\,q),\\
t&\equiv -\det Z,  &\qquad u&\equiv -\det Z    &&(\bmod\,r).
\end{aligned}
\]
Hence $\lambda\equiv\mu\equiv0$ modulo each of $p,q,r$.

It remains to show that $p,q,r$ are pairwise coprime. If a positive
integer $g$ divides both $p$ and $q$, then
$b\equiv0\pmod p$ and $b\equiv-\det Y\pmod q$ imply
$g\mid\det Y$, so $g=1$. Thus $\gcd(p,q)=1$. Similarly,
$c\equiv0\pmod p$ and $c\equiv\det Z\pmod r$ give
$\gcd(p,r)=1$, while $d\equiv0\pmod q$ and
$d\equiv-\det Z\pmod r$ give $\gcd(q,r)=1$.

Since $p,q,r$ are pairwise coprime and each of them divides both $\lambda$ and $\mu$, we obtain $pqr\mid\lambda,\,pqr\mid\mu$. Hence there exist $m,n\in\mathbb Z$ such that $\lambda=pqrm$ and $\mu=pqrn$. Substitution gives
\begin{equation}\label{eq:divisibility-identity}
p^2q^2r^2(3m^2+n^2)
=
4(t^2+tu+u^2-3).
\end{equation}
Since $p=\tr(X)$, $q=\tr(Y)$, and $r=\tr(Z)$, this is precisely
\[
\bigl(\tr(X)\tr(Y)\tr(Z)\bigr)^2(3m^2+n^2)
=
4(t^2+tu+u^2-3),
\]
which proves \textup{(iii)}.
\end{proof}

\begin{lemma}\label{lem:finite-reduced-data}
Define
\[
\mathcal A_-:=
\left\{
\begin{pmatrix}0&1\\-1&3\end{pmatrix},
\begin{pmatrix}2&-1\\-1&1\end{pmatrix},
\begin{pmatrix}2&1\\1&1\end{pmatrix}
\right\}
\
\text{and}
\
\mathcal A_+:=
\left\{
\begin{pmatrix}-2&1\\-3&1\end{pmatrix},
\begin{pmatrix}-2&3\\-1&1\end{pmatrix},
\begin{pmatrix}0&1\\-1&-1\end{pmatrix}
\right\}.
\]
Let $(X,Y,Z)\in\mathcal S$. Then exactly one of the following three cases occurs:
\[
\renewcommand{\arraystretch}{1.2}
\begin{array}{@{}ccc@{}}
\toprule
(\det X,\det Y,\det Z)
&
\bigl(\tr(X^2),\tr(Y^2),\tr(Z^2)\bigr)
&
\text{Condition on }X^2
\\
\midrule
(1,1,1)
&
(-1,-1,-2)
&
X^2=R
\\[2pt]
(-1,1,1)
&
(3,-1,2)
&
X^2\in\mathcal A_-
\\[2pt]
(1,-1,1)
&
(-1,3,2)
&
X^2\in\mathcal A_+
\\
\bottomrule
\end{array}
\]
\end{lemma}

\begin{proof}
Put $p=\tr(X)$, $q=\tr(Y)$, and $r=\tr(Z)$. By
Lemma~\ref{lem:trace-divisibility}, $p$ and $q$ are odd, whereas $r$
is even.

Suppose first that $r=0$. By Lemma~\ref{lem:trace-divisibility},
$(\det X,\det Y,\det Z)=(1,1,1)$ and $X^2=R$. Hence, by the definition
of $\mathcal S$, $Y^2=R^2$ and $Z^2=-R^3=-I$, so
\[
\bigl(\tr(X^2),\tr(Y^2),\tr(Z^2)\bigr)=(-1,-1,-2).
\]
This gives the first row of the table.

Assume now that $r\ne0$. Since all three traces are nonnegative,
$p,q\ge1$ and $r\ge2$. Set
$t=p^2-2\det X$ and $u=q^2-2\det Y$. By
\eqref{eq:trace-key-relation},
\[
r^2=p^2+q^2-2(\det X+\det Y-\det Z),
\]
and Lemma~\ref{lem:trace-divisibility} yields $m,n\in\mathbb Z$ such
that
\begin{equation}\label{eq:finite-divisibility}
p^2q^2r^2(3m^2+n^2)=4(t^2+tu+u^2-3).
\end{equation}

We claim that $\min\{p,q\}=1$. Suppose, to the contrary, that
$p,q\ge3$, and put
\[
\alpha=\det X,\qquad \beta=\det Y,\qquad \gamma=\det Z.
\]
Then
\[
r^2=p^2+q^2-2(\alpha+\beta-\gamma),\qquad
t=p^2-2\alpha,\qquad u=q^2-2\beta.
\]
Set
\[
L=p^2q^2r^2-4(t^2+tu+u^2-3).
\]
Writing $p^2=9+h$ and $q^2=9+k$, where $h,k\ge0$, a direct expansion
gives
\[
\begin{aligned}
L={}&
h^2k+hk^2+5h^2+5k^2\\
&+(32-2\alpha-2\beta+2\gamma)hk\\
&+(135-2\alpha-10\beta+18\gamma)h\\
&+(135-10\alpha-2\beta+18\gamma)k\\
&+466+54\alpha+54\beta+162\gamma-16\alpha\beta.
\end{aligned}
\]
The coefficients of $hk$, $h$, and $k$ are at least $26$, $105$,
and $105$, respectively, and the constant term is at least
$466-54-54-162-16=180$. Since $h,k\ge0$, it follows that $L>0$.

If $(m,n)\ne(0,0)$, then \eqref{eq:finite-divisibility} gives
$4(t^2+tu+u^2-3)\ge p^2q^2r^2$, contradicting $L>0$. If $m=n=0$,
then $t^2+tu+u^2-3=0$, whereas $t,u\ge7$, again a contradiction.
Hence
\[
\min\{p,q\}=1.
\]

We first consider $p=1$. Then \eqref{eq:trace-key-relation} becomes
\[
r^2-q^2=1-2(\det X+\det Y-\det Z).
\]
Since $q$ is positive and odd and $r$ is positive and even, factoring
the corresponding difference-of-squares equations for the eight
possible determinant triples gives exactly the following candidates:
\[
\begin{array}{@{}cc@{}}
\toprule
(\det X,\det Y,\det Z) & (q,r)\\
\midrule
(1,1,-1)   & (3,2)\\
(1,-1,1)   & (1,2)\\
(-1,1,1)   & (1,2)\\
(-1,-1,1)  & (3,4)\\
(-1,-1,-1) & (1,2)\\
\bottomrule
\end{array}
\]
the remaining three forcing $r=0$.

Three candidates are excluded by \eqref{eq:finite-divisibility}.
For $(\det X,\det Y,\det Z)=(1,1,-1)$, we have
$(p,q,r)=(1,3,2)$ and $(t,u)=(-1,7)$, which would require
$36\mid160$. For $(-1,-1,1)$, we have $(p,q,r)=(1,3,4)$ and
$(t,u)=(3,11)$, which would require $144\mid640$. Finally, for
$(-1,-1,-1)$, we have $(p,q,r)=(1,1,2)$ and $t=u=3$, so
\[
3m^2+n^2=24.
\]
Modulo $3$, this gives $3\mid n$; writing $n=3n_1$ yields
$m^2+3n_1^2=8$, hence $m^2\equiv2\pmod3$, a contradiction.

Thus, when $p=1$, the only possibilities are
\[
(\det X,\det Y,\det Z)\in\{(-1,1,1),(1,-1,1)\},
\qquad p=q=1,\qquad r=2.
\]
Since \eqref{eq:trace-key-relation} and
\eqref{eq:finite-divisibility} are invariant under
$(p,\det X,t)\leftrightarrow(q,\det Y,u)$, the same conclusion holds
when $q=1$. Hence, whenever $r\ne0$,
\[
(\det X,\det Y,\det Z)\in\{(-1,1,1),(1,-1,1)\},
\qquad p=q=1,\qquad r=2.
\]
Using $\tr(M^2)=\tr(M)^2-2\det M$, we obtain
\[
\bigl(\tr(X^2),\tr(Y^2),\tr(Z^2)\bigr)
=
\begin{cases}
(3,-1,2),&(\det X,\det Y,\det Z)=(-1,1,1),\\
(-1,3,2),&(\det X,\det Y,\det Z)=(1,-1,1).
\end{cases}
\]

It remains to determine $X^2$ in these two cases. Write
\[
A=X^2=\begin{pmatrix}a&b\\c&d\end{pmatrix}.
\]
Since $Y^2=AR$,
\[
AR=\begin{pmatrix}b&-a-b\\d&-c-d\end{pmatrix}.
\]
By Lemma~\ref{lem:square-root}, $t=\tr(A)$ and $u=\tr(AR)$, so $d=t-a$ and $c=a+b-t-u$. Thus $\det A=1$ gives
\begin{equation}\label{eq:finite-A-quadratic}
a^2+ab+b^2-ta-(t+u)b+1=0.
\end{equation}

If $(t,u)=(3,-1)$, then
\[
a^2+ab+b^2-3a-2b+1=0,
\]
whose discriminant in $a$ is $\Delta=-3b^2+2b+5$. Thus
$b\in\{-1,0,1\}$. For $b=-1$ one gets $a=2$; for $b=0$ there is no
integral solution; and for $b=1$ one gets $a=0$ or $2$. Since
$d=3-a$ and $c=a+b-2$,
\[
X^2\in
\left\{
\begin{pmatrix}0&1\\-1&3\end{pmatrix},
\begin{pmatrix}2&-1\\-1&1\end{pmatrix},
\begin{pmatrix}2&1\\1&1\end{pmatrix}
\right\}
=\mathcal A_-.
\]

If $(t,u)=(-1,3)$, then
\[
a^2+ab+b^2+a-2b+1=0,
\]
whose discriminant in $a$ is $\Delta=-3b^2+10b-3$. Hence
$b\in\{1,2,3\}$. For $b=1$ one gets $a=0$ or $-2$; for $b=2$ there
is no integral solution; and for $b=3$ one gets $a=-2$. Since
$d=-1-a$ and $c=a+b-2$,
\[
X^2\in
\left\{
\begin{pmatrix}-2&1\\-3&1\end{pmatrix},
\begin{pmatrix}-2&3\\-1&1\end{pmatrix},
\begin{pmatrix}0&1\\-1&-1\end{pmatrix}
\right\}
=\mathcal A_+.
\]

Thus every $(X,Y,Z)\in\mathcal S$ belongs to one of the three rows in
the statement. Since their determinant triples are distinct, the three
cases are mutually exclusive. This proves the lemma.
\end{proof}

We next give an explicit description of the set $\mathcal S$.

\begin{lemma}\label{lem:S-explicit}
Let $H=\begin{pmatrix}-1&1\\0&1\end{pmatrix}$ and $\mathcal T_{J}=\{PJP^{-1}:P\in G\}$. Then
\begin{align}
\mathcal S
&={}\{(-R^2,-R,T):T\in\mathcal T_{J}\}\label{eq:S-first}\\
&\quad\cup\left\{(R^{-k}UR^k,R^{-k}VR^k,R^{-k}WR^k):k=0,1,2\right\}\label{eq:S-second}\\
&\quad\cup\left\{(H^{-1}R^{-k}VR^kH,H^{-1}R^{-k}UR^kH,H^{-1}R^{-k}WR^kH):k=0,1,2\right\}.\label{eq:S-third}
\end{align}
\end{lemma}

\begin{proof}
We prove the two inclusions separately. Let $(X,Y,Z)\in\mathcal S$.
By Lemma~\ref{lem:finite-reduced-data}, exactly one of its three cases
occurs. In each case, the identity
\[
\tr(M^2)=\tr(M)^2-2\det M,
\]
together with the nonnegativity of the traces, determines
$\tr(X),\tr(Y)$, and $\tr(Z)$ uniquely.

In the first case,
\[
(\det X,\det Y,\det Z)=(1,1,1),\qquad X^2=R,
\]
and the square-trace data in Lemma~\ref{lem:finite-reduced-data} give
\[
(\tr(X),\tr(Y),\tr(Z))=(1,1,0).
\]
Since $Y^2=R^2$ and $Z^2=-R^3=-I$, Lemma~\ref{lem:square-root}
gives
\[
X=R+I=-R^2,\qquad Y=R^2+I=-R.
\]
Moreover, Lemma~\ref{lem:minusI-conjugacy} shows that
$Z\in\mathcal T_J$. Hence $(X,Y,Z)$ belongs to the set in
\eqref{eq:S-first}.

In the second case,
\[
(\det X,\det Y,\det Z)=(-1,1,1),\qquad X^2\in\mathcal A_-,
\]
and
\[
(\tr(X),\tr(Y),\tr(Z))=(1,1,2).
\]
Set $A=X^2$. Lemma~\ref{lem:square-root}, applied successively to
$X,Y$, and $Z$, yields
\begin{equation}\label{eq:S-roots-minus}
X=A-I,\qquad
Y=AR+I,\qquad
Z=\frac{-AR^2+I}{2}.
\end{equation}
For $k=0,1,2$, define
\[
A_k:=R^{-k}U^2R^k.
\]
A direct computation gives
\begin{equation}\label{eq:S-Ak-data}
\mathcal A_-=\{A_0,A_1,A_2\},
\qquad
U^2-I=U,\quad
U^2R+I=V,\quad
\frac{-U^2R^2+I}{2}=W.
\end{equation}
Because every power of $R$ commutes with $R$, \eqref{eq:S-Ak-data}
implies
\[
A_k-I=R^{-k}UR^k,\qquad
A_kR+I=R^{-k}VR^k,\qquad
\frac{-A_kR^2+I}{2}=R^{-k}WR^k.
\]
Thus, if $A=A_k$, \eqref{eq:S-roots-minus} places $(X,Y,Z)$ in
the set in \eqref{eq:S-second}.

In the third case,
\[
(\det X,\det Y,\det Z)=(1,-1,1),\qquad X^2\in\mathcal A_+,
\]
and again
\[
(\tr(X),\tr(Y),\tr(Z))=(1,1,2).
\]
Set $B=X^2$. Lemma~\ref{lem:square-root} gives
\begin{equation}\label{eq:S-roots-plus}
X=B+I,\qquad
Y=BR-I,\qquad
Z=\frac{-BR^2+I}{2}.
\end{equation}
Let
\[
\varphi(C):=H^{-1}CH.
\]
A direct computation gives
\[
\varphi(R)=R^2,\qquad \varphi(R^2)=R,
\]
and, with the matrices $A_k$ defined above,
\begin{equation}\label{eq:S-Bk-data}
\mathcal A_+=\{B_0,B_1,B_2\},
\qquad
B_k:=\varphi(A_kR).
\end{equation}
Since $B_k=\varphi(A_k)R^2$ and $R^3=I$, we have
\[
B_kR=\varphi(A_k),\qquad
B_kR^2=\varphi(A_kR^2).
\]
Consequently,
\[
\begin{aligned}
B_k+I&=\varphi(A_kR+I),\\
B_kR-I&=\varphi(A_k-I),\\
\frac{-B_kR^2+I}{2}
&=\varphi\left(\frac{-A_kR^2+I}{2}\right).
\end{aligned}
\]
Combining these identities with \eqref{eq:S-Ak-data} and
\eqref{eq:S-roots-plus}, we obtain
\[
\begin{aligned}
X&=H^{-1}R^{-k}VR^kH,\\
Y&=H^{-1}R^{-k}UR^kH,\\
Z&=H^{-1}R^{-k}WR^kH.
\end{aligned}
\]
Thus $(X,Y,Z)$ belongs to the set in \eqref{eq:S-third}. This proves
that $\mathcal S$ is contained in the union on the right-hand side of
the asserted equality.

For the reverse inclusion, first let $T\in\mathcal T_J$. Then
$T=PJP^{-1}$ for some $P\in G$, so $T^2=-I$ and $\tr(T)=0$. Hence
\[
(-R^2)^2=R,\qquad
(-R)^2=R^2,\qquad
T^2=-I=-R^3,
\]
and the traces of $-R^2,-R$, and $T$ are $1,1$, and $0$,
respectively. Therefore every triple in \eqref{eq:S-first} belongs to
$\mathcal S$.

Next, a direct computation gives
\begin{equation}\label{eq:S-UVW-squares}
V^2=U^2R,\qquad
W^2=-U^2R^2,\qquad
(\tr(U),\tr(V),\tr(W))=(1,1,2).
\end{equation}
For
\[
A_k=R^{-k}U^2R^k\in\SL_2(\mathbb Z),
\]
the definition of $A_k$ and the two square identities in
\eqref{eq:S-UVW-squares} give
\[
\begin{aligned}
(R^{-k}UR^k)^2&=A_k,\\
(R^{-k}VR^k)^2&=A_kR,\\
(R^{-k}WR^k)^2&=-A_kR^2.
\end{aligned}
\]
Trace is invariant under conjugation, so every triple in
\eqref{eq:S-second} belongs to $\mathcal S$.

Finally, let
\[
B_k=\varphi(A_kR)\in\SL_2(\mathbb Z).
\]
The squares of the three components of the corresponding triple in
\eqref{eq:S-third} are
\[
\varphi(A_kR)=B_k,\qquad
\varphi(A_k)=B_kR,\qquad
-\varphi(A_kR^2)=-B_kR^2,
\]
respectively. These components are conjugate to $V,U$, and $W$, so
their traces are $1,1$, and $2$. Hence every triple in
\eqref{eq:S-third} also belongs to $\mathcal S$. This proves the
reverse inclusion.
\end{proof}

The second and third families in Lemma~\ref{lem:S-explicit} each give
rise to a single $G$-orbit. Indeed, simultaneous conjugation by
$R^{-k}$ gives
\[
\bigcup_{k=0}^2
\mathcal O_G(R^{-k}UR^k,R^{-k}VR^k,R^{-k}WR^k)
=
\mathcal O_G(U,V,W).
\]
Similarly, every triple in the third family is obtained from
$(V,U,W)$ by simultaneous conjugation by $H^{-1}R^{-k}$, and hence
that family gives rise to $\mathcal O_G(V,U,W)$.

It remains to determine the distinct $G$-orbits arising from the first
family. In this case, the first two components are fixed, whereas the
third component ranges over all integral square roots of $-I$.
Consequently, the relevant conjugating matrices must lie in the
centralizer of $R$.

We first determine the centralizer of $R$ in $G$.

\begin{lemma}\label{lem:centralizer-R}
The centralizer of $R$ in $G$ is $C_G(R)=\{\pm I,\pm R,\pm R^2\}$.
\end{lemma}

\begin{proof}
Write
\[
P=\begin{pmatrix}a&b\\c&d\end{pmatrix}\in G.
\]
A comparison of the entries of $PR$ and $RP$ shows that $PR=RP$ if
and only if
\[
c=-b,\qquad d=a+b.
\]
Hence
\[
P=\begin{pmatrix}a&b\\-b&a+b\end{pmatrix},
\qquad
\det P=a^2+ab+b^2.
\]
Since $P\in G$, we have $\det P\in\{\pm1\}$. On the other hand, $4\det P=(2a+b)^2+3b^2\ge0$. It follows that $\det P=1$, and therefore $(2a+b)^2+3b^2=4$. In particular, $3b^2\le4$, so $b\in\{0,\pm1\}$. If $b=0$, then
$a=\pm1$. If $b=1$, then $a\in\{0,-1\}$, whereas if $b=-1$, then
$a\in\{0,1\}$. The resulting six matrices are precisely $\pm I,\, \pm R,\,\pm R^2$.
\end{proof}

We next consider the natural action of the cyclic subgroup $\langle R\rangle=\{I,R,R^2\}$ on $\mathcal T_J$ by integral conjugation. For
$Q\in\langle R\rangle$ and $T\in\mathcal T_J$, define
\[
Q\cdot T:=QTQ^{-1}.
\]
This map is well defined. Indeed, if $T=PJP^{-1}\in\mathcal T_J$ for some $P\in G$, then
\[
Q\cdot T
=
QPJP^{-1}Q^{-1}
=
(QP)J(QP)^{-1}.
\]
Since $Q,P\in G$, we have $QP\in G$, and hence $Q\cdot T\in\mathcal T_J$. Finally, the identity element acts trivially and
\[
(Q_1Q_2)\cdot T
=
Q_1\cdot(Q_2\cdot T),
\qquad
Q_1,Q_2\in\langle R\rangle,
\]
so this indeed defines a group action of $\langle R\rangle$ on
$\mathcal T_J$. For $T\in\mathcal T_J$, we denote its orbit under
this action by
\[
\begin{aligned}
\mathcal O_{\langle R\rangle}(T)
&:=
\{QTQ^{-1}:Q\in\langle R\rangle\}
=\{R^kTR^{-k}:k=0,1,2\}.
\end{aligned}
\]

The conjugation action of $\langle R\rangle$ on $\mathcal T_J$
captures precisely the residual equivalence among the triples in the
first family, as the following lemma shows.

\begin{lemma}\label{lem:first-orbit-equivalence}
Let $T_1,T_2\in\mathcal T_J$. Then $\mathcal O_G(-R^2,-R,T_1)=\mathcal O_G(-R^2,-R,T_2)$ if and only if $\mathcal O_{\langle R\rangle}(T_1)=\mathcal O_{\langle R\rangle}(T_2).$
\end{lemma}

\begin{proof}
For the conjugation action of $\langle R\rangle$ on $\mathcal T_J$,
two orbits coincide if and only if one of their elements belongs to
the orbit of the other. Hence $\mathcal O_{\langle R\rangle}(T_1)=\mathcal O_{\langle R\rangle}(T_2)$
if and only if $T_2=R^kT_1R^{-k}$ for some $k\in\{0,1,2\}$.

Suppose first that this condition holds. Since $R^k$ commutes with both $R$ and $R^2$, simultaneous conjugation by $R^k$ sends $(-R^2,-R,T_1)$ to $(-R^2,-R,T_2)$. Therefore,
\[
\mathcal O_G(-R^2,-R,T_1)
=
\mathcal O_G(-R^2,-R,T_2).
\]

Conversely, suppose that
$
\mathcal O_G(-R^2,-R,T_1)
=
\mathcal O_G(-R^2,-R,T_2)$. Then there exists $P\in G$ such that
\[
P(-R^2)P^{-1}=-R^2,\qquad
P(-R)P^{-1}=-R,\qquad
PT_1P^{-1}=T_2.
\]
The second identity implies that $P\in C_G(R)$. By
Lemma~\ref{lem:centralizer-R}, we may write
\[
P=(\varepsilon I)R^k
\]
for some $\varepsilon\in\{\pm1\}$ and $k\in\{0,1,2\}$. Since
$\varepsilon I$ is central in $G$, it follows that
\[
T_2
=
PT_1P^{-1}
=
R^kT_1R^{-k}.
\]
Thus $T_2\in\mathcal O_{\langle R\rangle}(T_1)$, and consequently
$
\mathcal O_{\langle R\rangle}(T_1)
=
\mathcal O_{\langle R\rangle}(T_2).
$
\end{proof}

We next give an explicit parametrization of $\mathcal T_J$.

\begin{lemma}\label{lem:T-parametrization}
We have
\[
\mathcal T_J
=
\left\{
\begin{pmatrix}r&s\\t&-r\end{pmatrix}:
r,s,t\in\mathbb Z,\  r^2+st=-1
\right\}.
\]
\end{lemma}

\begin{proof}
Let $T\in\mathcal T_J$. By definition, $T=PJP^{-1}$ for some
$P\in G$. Since trace and determinant are invariant under conjugation,
we have $\tr(T)=\tr(J)=0$ and $\det T=\det J=1$. Hence
\[
T=\begin{pmatrix}r&s\\t&-r\end{pmatrix}
\]
for some $r,s,t\in\mathbb Z$, and the condition $\det T=1$ gives
$-r^2-st=1$, or equivalently, $r^2+st=-1$.

Conversely, let
\[
T=\begin{pmatrix}r&s\\t&-r\end{pmatrix},
\qquad r^2+st=-1.
\]
Then $\tr(T)=0$ and $\det T=-r^2-st=1$, so
$T\in\SL_2(\mathbb Z)\subseteq G$. By the Cayley--Hamilton identity,
$T^2-\tr(T)T+(\det T)I=0$, and hence $T^2=-I$. By
Lemma~\ref{lem:minusI-conjugacy}, there exists $Q\in G$ such that
$Q^{-1}TQ=J$. Therefore, $T=QJQ^{-1}\in\mathcal T_J$.
\end{proof}

To describe the $\langle R\rangle$-orbits in $\mathcal T_J$, we
introduce a cyclic action on a set of integer triples. Let
\[
\Omega:=\{(a,b,c)\in\mathbb Z^3:ab+bc+ca=1\}.
\]
The defining equation is invariant under cyclic permutation, so the
map
\[
\sigma:\Omega\longrightarrow\Omega,
\qquad
\sigma(a,b,c):=(b,c,a),
\]
is well defined. Moreover,
\[
\sigma^2(a,b,c)=(c,a,b),
\qquad
\sigma^3(a,b,c)=(a,b,c),
\]
and hence $\sigma^3=\operatorname{id}_{\Omega}$. It follows that
$\sigma$ is a bijection with inverse $\sigma^2$, and
\[
\langle\sigma\rangle
=
\{\operatorname{id}_{\Omega},\sigma,\sigma^2\}
\]
is a cyclic subgroup of the group of all bijections of $\Omega$ under
composition.

This subgroup acts naturally on $\Omega$ by evaluation: for
$\tau\in\langle\sigma\rangle$ and $x\in\Omega$, define
\[
\tau\cdot x:=\tau(x).
\]
Indeed,
\[
\operatorname{id}_{\Omega}\cdot x=x,
\qquad
(\tau_1\tau_2)\cdot x
=
\tau_1\cdot(\tau_2\cdot x)
\]
for all $\tau_1,\tau_2\in\langle\sigma\rangle$ and $x\in\Omega$.
Thus this defines a group action of $\langle\sigma\rangle$ on
$\Omega$. The orbit of $(a,b,c)\in\Omega$ under this action is
\[
\begin{aligned}
\mathcal O_{\langle\sigma\rangle}(a,b,c)
&:=
\{\sigma^k(a,b,c):k=0,1,2\}\\
&=
\{(a,b,c),(b,c,a),(c,a,b)\}.
\end{aligned}
\]

We denote the orbit spaces of the conjugation action of
$\langle R\rangle$ on $\mathcal T_J$ and the above action of
$\langle\sigma\rangle$ on $\Omega$ by
\[
\begin{aligned}
\mathcal T_J/\langle R\rangle
&:=
\left\{
\mathcal O_{\langle R\rangle}(T):T\in\mathcal T_J
\right\},\\
\Omega/\langle\sigma\rangle
&:=
\left\{
\mathcal O_{\langle\sigma\rangle}(a,b,c):(a,b,c)\in\Omega
\right\},
\end{aligned}
\]
respectively. The following lemma shows that $\Phi$ intertwines these
two cyclic actions and consequently identifies their orbit spaces.

\begin{lemma}\label{lem:Phi-bijection}
Define
\[
\begin{aligned}
\Phi:\mathcal T_J&\longrightarrow\Omega,\\
\begin{pmatrix}r&s\\t&-r\end{pmatrix}
&\longmapsto(r,t-r,-s-r).
\end{aligned}
\]
Then the following statements hold:
\begin{enumerate}
\item[\textup{(i)}]
The map $\Phi$ is a bijection, and its inverse is given by
\[
\Phi^{-1}(a,b,c)
=
T(a,b,c)
:=
\begin{pmatrix}
a&-a-c\\
a+b&-a
\end{pmatrix}.
\]

\item[\textup{(ii)}]
For every $T\in\mathcal T_J$ and $k\in\{0,1,2\}$,
\begin{equation}\label{eq:Phi-equivariance}
\Phi(R^kTR^{-k})
=
\sigma^k\bigl(\Phi(T)\bigr).
\end{equation}

\item[\textup{(iii)}]
For every $T\in\mathcal T_J$,
\[
\Phi\bigl(\mathcal O_{\langle R\rangle}(T)\bigr)
=
\mathcal O_{\langle\sigma\rangle}\bigl(\Phi(T)\bigr).
\]

\item[\textup{(iv)}]
The induced map
\[
\begin{aligned}
\overline{\Phi}:
\mathcal T_J/\langle R\rangle
&\longrightarrow
\Omega/\langle\sigma\rangle,\\
\mathcal O_{\langle R\rangle}(T)
&\longmapsto
\mathcal O_{\langle\sigma\rangle}\bigl(\Phi(T)\bigr)
\end{aligned}
\]
is a bijection.
\end{enumerate}
\end{lemma}

\begin{proof}
For \textup{(i)}, write
\[
M=
\begin{pmatrix}
r&s\\
t&-r
\end{pmatrix}
\in\mathcal T_J.
\]
By Lemma~\ref{lem:T-parametrization}, $r^2+st=-1$. Setting
$(a,b,c)=(r,t-r,-s-r)$, we obtain
\[
ab+bc+ca=-r^2-st=1.
\]
Thus $\Phi(M)\in\Omega$, so $\Phi$ is well defined.

Conversely, let $(a,b,c)\in\Omega$. Since
\[
a^2+(-a-c)(a+b)
=
-(ab+bc+ca)
=
-1,
\]
Lemma~\ref{lem:T-parametrization} implies that
$T(a,b,c)\in\mathcal T_J$. Direct substitution gives
\[
\Phi(T(a,b,c))=(a,b,c),
\qquad
T(\Phi(M))=M.
\]
Hence $\Phi$ is a bijection with the stated inverse.

For \textup{(ii)}, a direct calculation gives
\[
RMR^{-1}
=
\begin{pmatrix}
t-r&-t\\
t-2r-s&r-t
\end{pmatrix},
\]
and therefore
\[
\Phi(RMR^{-1})
=
(t-r,-s-r,r)
=
\sigma\bigl(\Phi(M)\bigr).
\]
This proves the assertion for $k=1$. The case $k=0$ is immediate,
and iterating the preceding identity gives the case $k=2$.

For \textup{(iii)}, it follows from \textup{(ii)} that
\[
\begin{aligned}
\Phi\bigl(\mathcal O_{\langle R\rangle}(T)\bigr)
&=
\{\Phi(R^kTR^{-k}):k=0,1,2\}\\
&=
\{\sigma^k(\Phi(T)):k=0,1,2\}\\
&=
\mathcal O_{\langle\sigma\rangle}\bigl(\Phi(T)\bigr).
\end{aligned}
\]

For \textup{(iv)}, part \textup{(iii)} shows that the definition of
$\overline{\Phi}$ is independent of the choice of representative, so
$\overline{\Phi}$ is well defined. Suppose that
\[
\overline{\Phi}\bigl(\mathcal O_{\langle R\rangle}(T_1)\bigr)
=
\overline{\Phi}\bigl(\mathcal O_{\langle R\rangle}(T_2)\bigr).
\]
Then $\Phi(T_2)=\sigma^k(\Phi(T_1))$ for some
$k\in\{0,1,2\}$. By \textup{(ii)},
\[
\Phi(T_2)=\Phi(R^kT_1R^{-k}),
\]
so the injectivity of $\Phi$ gives $T_2=R^kT_1R^{-k}$. Hence
\[
\mathcal O_{\langle R\rangle}(T_1)
=
\mathcal O_{\langle R\rangle}(T_2),
\]
and therefore $\overline{\Phi}$ is injective.

Finally, given $(a,b,c)\in\Omega$, let
$T=\Phi^{-1}(a,b,c)$. Then
\[
\overline{\Phi}\bigl(\mathcal O_{\langle R\rangle}(T)\bigr)
=
\mathcal O_{\langle\sigma\rangle}(a,b,c),
\]
so $\overline{\Phi}$ is surjective. Therefore $\overline{\Phi}$ is a
bijection.
\end{proof}

To choose a canonical representative from each
$\langle\sigma\rangle$-orbit, define
\[
\Omega_0
:=
\{(a,b,c)\in\Omega:a\le b,\ a<c\}.
\]

\begin{lemma}\label{lem:Omega0-transversal}
For every $u\in\Omega$,
\[
\left|
\mathcal O_{\langle\sigma\rangle}(u)\cap\Omega_0
\right|
=
1.
\]
\end{lemma}

\begin{proof}
Let $u=(a,b,c)\in\Omega$ and set $m=\min\{a,b,c\}$. The three
coordinates cannot all be equal. Indeed, if $a=b=c=m$, then the
defining equation of $\Omega$ would give $3m^2=1$, which is
impossible for $m\in\mathbb Z$. Thus the minimum occurs either once
or twice.

Suppose first that the minimum occurs exactly once. There is then a
unique cyclic rotation of $u$ of the form $(m,x,y)$. Since $m<x$ and
$m<y$, this rotation belongs to $\Omega_0$. The rotation $(x,y,m)$
fails the condition that the first coordinate be strictly smaller
than the third, while $(y,m,x)$ fails the condition that the first
coordinate be at most the second. Hence exactly one element of the
orbit lies in $\Omega_0$.

Suppose now that the minimum occurs twice. Writing the remaining
coordinate as $M$, we have $M>m$, and the orbit consists of
\[
(m,m,M),\qquad (m,M,m),\qquad (M,m,m).
\]
The first triple belongs to $\Omega_0$, the second fails the strict
inequality between the first and third coordinates, and the third
fails the weak inequality between the first and second coordinates.
Thus the orbit again meets $\Omega_0$ in exactly one point.
\end{proof}

Combining Lemmas~\ref{lem:Phi-bijection} and
\ref{lem:Omega0-transversal}, we obtain an explicit transversal for
the $\langle R\rangle$-orbits in $\mathcal T_J$.

\begin{corollary}\label{cor:T-transversal}
Define
\[
\begin{aligned}
\mathcal T
&:=
\Phi^{-1}(\Omega_0)
=
\left\{
\begin{pmatrix}
a&-a-c\\
a+b&-a
\end{pmatrix}
:
a,b,c\in\mathbb Z,\
ab+bc+ca=1,\
a\le b,\
a<c
\right\}.
\end{aligned}
\]
Then $\mathcal T$ is a complete set of representatives for the
$\langle R\rangle$-orbits in $\mathcal T_J$. More precisely, for
every $S\in\mathcal T_J$,
\[
\left|
\mathcal O_{\langle R\rangle}(S)\cap\mathcal T
\right|
=
1.
\]
Equivalently,
\[
\mathcal T_J
=
\mathop{\bigsqcup}\limits_{T\in\mathcal T}
\mathcal O_{\langle R\rangle}(T).
\]
\end{corollary}

\begin{proof}
By Lemma~\ref{lem:Phi-bijection}, the bijection $\Phi$ maps each
$\langle R\rangle$-orbit in $\mathcal T_J$ bijectively onto the
corresponding $\langle\sigma\rangle$-orbit in $\Omega$. By
Lemma~\ref{lem:Omega0-transversal}, every
$\langle\sigma\rangle$-orbit in $\Omega$ meets $\Omega_0$ in exactly
one point. Therefore, every $\langle R\rangle$-orbit in
$\mathcal T_J$ meets $\Phi^{-1}(\Omega_0)=\mathcal T$ in exactly one
point.
\end{proof}

\section{Proof of Theorem~\ref{thm:orbit-main}}\label{sec:main-proof}

In this section, we give the proof of Theorem~\ref{thm:orbit-main}.

\begin{proof}[Proof of Theorem~\ref{thm:orbit-main}]
By Lemmas~\ref{lem:orbit-reduction} and \ref{lem:S-explicit}, and by
the orbit identifications following Lemma~\ref{lem:S-explicit}, we
obtain
\begin{equation}\label{eq:precanonical-orbits}
\mathcal M
=
\left(
\bigcup_{T\in\mathcal T_J}
\mathcal O_G(-R^2,-R,T)
\right)
\sqcup
\mathcal O_G(U,V,W)
\sqcup
\mathcal O_G(V,U,W).
\end{equation}
The union indexed by $\mathcal T_J$ is an ordinary union, since
distinct parameters may determine the same $G$-orbit. The three terms
in \eqref{eq:precanonical-orbits} are nevertheless pairwise disjoint:
their determinant triples are, respectively,
\[
(1,1,1),\qquad
(-1,1,1),\qquad
(1,-1,1),
\]
and determinant is invariant under simultaneous conjugation.

It remains to eliminate the duplicate parametrizations in the first
term. By Lemma~\ref{lem:first-orbit-equivalence}, two parameters
$T_1,T_2\in\mathcal T_J$ determine the same $G$-orbit of triples if
and only if they belong to the same $\langle R\rangle$-orbit.
Corollary~\ref{cor:T-transversal} states that $\mathcal T$ contains
exactly one representative from each $\langle R\rangle$-orbit in
$\mathcal T_J$. Consequently,
\[
\bigcup_{S\in\mathcal T_J}
\mathcal O_G(-R^2,-R,S)
=
\mathop{\sqcup}\limits_{T\in\mathcal T}
\mathcal O_G(-R^2,-R,T).
\]
Substituting this identity into \eqref{eq:precanonical-orbits} gives
\eqref{eq:main-orbit-decomposition}.
\end{proof}

\end{document}